\documentclass[pdflatex,sn-mathphys-num]{sn-jnl}% Math and Physical Sciences Numbered Reference Style
\usepackage{graphicx}%
\usepackage{multirow}%
\usepackage{amsmath,amssymb,amsfonts}%
\usepackage{amsthm}%
\usepackage{mathrsfs}%
\usepackage[title]{appendix}%
\usepackage{xcolor}%
\usepackage{textcomp}%
\usepackage{manyfoot}%
\usepackage{booktabs}%
\usepackage{algorithm}%
\usepackage{algorithmicx}%
\usepackage{algpseudocode}%
\usepackage{listings}%
\theoremstyle{thmstyleone}%
\newtheorem{theorem}{Theorem}%  meant for continuous numbers
\newtheorem{proposition}[theorem]{Proposition}% 

\theoremstyle{thmstyletwo}%
\newtheorem{example}{Example}%
\newtheorem{remark}{Remark}%
\newtheorem{lemma}[theorem]{Lemma}
\theoremstyle{thmstylethree}%
\newtheorem{definition}{Definition}%
\newtheorem{corollary}[theorem]{Corollary}

\newcommand{\R}{\mathbb{R}}

\newcommand{\eps}{\varepsilon}
\newcommand{\norm}[1]{\left\|#1\right\|}
\newcommand{\diam}{\operatorname{diam}}
\newcommand{\dist}{\operatorname{dist}}

\newcommand{\JI}{J_{I}^{\eps}}
\newcommand{\JB}{J_{B}^{\eps}}
\newcommand{\CNJ}{C_{\rm NJ}}
\newcommand{\CZ}{C_Z}

\begin{document}

\title[Article Title]{James-type constants for Birkhoff--James orthogonality and approximate isosceles}


\author[1]{\fnm{Zhiyao} \sur{Fang}}
\email{fzy73@stu.aqnu.edu.cn}

\author[1]{\fnm{Ying} \sur{Xu}}
\email{Y26220046@stu.aqnu.edu.cn}

\author*[1]{\fnm{Qi} \sur{Liu}}
\email{liuq67@aqnu.edu.cn}

\author[2]{\fnm{Zhaohui} \sur{Gu}}
\email{zhgugz@163.com}

\author[3]{\fnm{Yongjin} \sur{Li}}
\email{stslyj@mail.sysu.edu.cn}

\affil[1]{%
	\orgdiv{School of Mathematics and Statistics},
	\orgname{Anqing Normal University},
	\orgaddress{
		\city{Anqing},
		\postcode{246133},
		\state{Anhui},
		\country{China}
	}
}

\affil[2]{%
	\orgdiv{School of Mathematics and Statistics},
	\orgname{Guangdong University of Foreign Studies},
	\orgaddress{
		\city{Guangzhou},
		\postcode{510006},
		\country{China}
	}
}

\affil[3]{%
	\orgdiv{Department of Mathematics},
	\orgname{Sun Yat-sen University},
	\orgaddress{
		\city{Guangzhou},
		\postcode{510275},
		\state{Guangdong},
		\country{China}
	}
}


\abstract{We introduce two James-type constants associated with approximate
	isosceles and approximate Birkhoff--James orthogonality in real normed
	spaces. The isosceles family coincides with the classical James constant
	of Gao and Lau [J. Austral. Math. Soc. Ser. A 48 (1990), 101--112],
	while the Birkhoff family extends the orthogonal James constant of
	Baronti and Papini [Constr. Math. Anal. 5 (2022), no.~1, 37--45].
	We show that approximate orthogonality constraints can, in certain
	settings, recover classical geometric constants exactly. In particular,
	the approximate isosceles family collapses to the classical James
	constant, while the approximate Birkhoff--James family forms a genuine
	parameter-dependent interpolation between the orthogonal and classical
	James constants.We also compare the Birkhoff profile with several classical geometric constants and study how it changes with the parameter.Several examples are also given to illustrate the relation between geometric constants and approximate orthogonality.}

\keywords{Birkhoff-James orthogonality; isosceles orthogonality; approximate orthogonality; James constant;  fixed point property.}


\pacs[MSC Classification]{46B20.}

\maketitle

\section{Introduction}\label{sec:introduction}

Orthogonality in a real normed space is no longer unique. Two of the
most widely used substitutes for Hilbert space orthogonality are
isosceles orthogonality and Birkhoff--James orthogonality. For vectors
$x,y$ in a real normed space $X$, isosceles orthogonality, introduced
by James \cite{James1945}, is defined by
\[
x\perp_I y
\quad\text{if and only if}\quad
\norm{x+y}=\norm{x-y},
\]
whereas Birkhoff--James orthogonality, originating with Birkhoff
\cite{Birkhoff1935} and further developed by James
\cite{James1945,James1947}, is defined by
\[
x\perp_B y
\quad\text{if and only if}\quad
\norm{x+\lambda y}\geq\norm{x}
\quad\text{for every }\lambda\in\R.
\]

We shall use the following approximate versions of these two
orthogonality relations. Following the approximate isosceles
orthogonality considered in \cite{ChmielinskiWojcik2010}, for
$0\leq\eps<1$ we say that $x$ is $I$-$\eps$-orthogonal to $y$, and
write $x\perp_{I,\eps}y$, if
\[
(1-\eps)\norm{x-y}
\leq \norm{x+y}
\leq (1+\eps)\norm{x-y}.
\]
When $\eps=0$, this reduces to the classical isosceles orthogonality.

Following Dragomir \cite{Dragomir1991}, for $0\leq\eps<1$ we say that
$x$ is $B$-$\eps$-orthogonal to $y$, and write
$x\perp_{B,\eps}y$, if
\[
\norm{x+\lambda y}\geq(1-\eps)\norm{x}
\quad\text{for every }\lambda\in\R.
\]
When $\eps=0$, this reduces to the classical Birkhoff--James
orthogonality.

Approximate forms of orthogonality in normed spaces have been studied
extensively; see, for example,
\cite{Dragomir1991,Chmielinski2005,Chmielinski2023,
	ChmielinskiWojcik2010,Dehghani2026}.

We shall also use the von Neumann--Jordan constant
\cite{Clarkson1937},
\[
\CNJ(X)
=
\sup_{(x,y)\neq(0,0)}
\frac{\norm{x+y}^{2}+\norm{x-y}^{2}}
{2\bigl(\norm{x}^{2}+\norm{y}^{2}\bigr)},
\]
and the Zbaganu constant \cite{Zbaganu2002},
\[
\CZ(X)
=
\sup_{(x,y)\neq(0,0)}
\frac{\norm{x+y}\norm{x-y}}
{\norm{x}^{2}+\norm{y}^{2}}.
\]
These constants quantify, in different ways, the departure of the norm
from Hilbert space geometry and will be used below to estimate the
approximate endpoint constants.

We further denote by $\delta_X$ the modulus of convexity of $X$:
\[
\delta_X(t)
=
\inf\left\{
1-\frac{\norm{x+y}}{2}:
x,y\in S_X,\ \norm{x-y}\geq t
\right\},
\qquad 0\leq t\leq2,
\]
and
\[
\rho_X(t)
=
\sup\left\{
\frac{\norm{x+ty}+\norm{x-ty}}{2}-1:
x,y\in S_X
\right\},
\qquad t\geq0.
\]

We shall use the standard bounds
\[
\sqrt{2}\leq J(X)\leq2,
\qquad
1\leq \CZ(X)\leq \CNJ(X)\leq2,
\]
together with the classical characterization
\[
X\ \text{is uniformly non-square}
\quad\Longleftrightarrow\quad
J(X)<2.
\]
These facts, as well as the basic properties of $\delta_X$, are standard;
see, for example, \cite{James1964,Kato2001}.

In recent years, several orthogonality-based geometric constants have
been introduced to quantify different aspects of the geometry of Banach
spaces. Representative examples include
\[
\begin{aligned}
	D(X)
	&=\inf\left\{
	\inf_{\lambda\in\mathbb R}\norm{x+\lambda y}:
	x,y\in S_X,\ x\perp_I y
	\right\},\\
	D_{\varepsilon}(X)
	&=\inf\left\{
	\inf_{\lambda\in\mathbb R}
	\bigl(\norm{x+\lambda y}+\varepsilon|\lambda|\bigr):
	x,y\in S_X,\ x\perp_I^{\varepsilon} y
	\right\},\\
	D'(X)
	&=\sup\left\{
	\bigl|\norm{x+y}-\norm{x-y}\bigr|:
	x,y\in S_X,\ x\perp_B y
	\right\},\\
	J_{\perp}(X)
	&=\sup\left\{
	\min\{\norm{x+y},\norm{x-y}\}:
	x,y\in S_X,\ x\perp_B y
	\right\},\\
	BS(X)
	&=\sup\left\{
	\frac{\norm{x+y}}{\norm{x-y}}:
	x,y\in S_X,\ x\perp_B y
	\right\},\\
	BR(X)
	&=\sup\left\{
	\frac{\bigl|\norm{x+\lambda y}-\norm{x-\lambda y}\bigr|}{\lambda}:
	x,y\in S_X,\ x\perp_B y,\ \lambda>0
	\right\},\\	
	L_X(\lambda)
	&=\sup_{\substack{x,y\in X,\ (x,y)\neq(0,0)\\ x\perp_I y}}
	\frac{\norm{\lambda x+(1-\lambda)y}^2+\norm{(1-\lambda)x+\lambda y}^2}
	{\norm{x+y}^2},
	\qquad 0\leq \lambda<\frac12,\\
	A'_{\lambda-\mu}(X)
	&=\sup\left\{
	\frac{\norm{\lambda x+\mu y}+\norm{\mu x-\lambda y}}{2}:
	x,y\in S_X,\ x\perp_I y
	\right\},
	\qquad \lambda,\mu>0,\\
	\theta_{\diamond}(X)
	&=\sup\left\{
	\frac{\norm{\norm{x+y}x-(x+y)}}{\norm{x+y}}:
	x,y\in S_X,\ x\perp_{\diamond}y
	\right\},
	\qquad \diamond\in\{I,B\}.
\end{aligned}
\]
See, for example,
See, for example,
\cite{JiWu2006,Dehghani2026,JiJiaWu2007,PapiniWu2013,
	BarontiPapini2022,HeRaoWangLiuLi2026,BiLiuLi2026,
	NiLiuWangXiaWang2025,NiLiuZhou2025}.

The present paper follows this viewpoint from an approximate perspective:
instead of fixing an exact orthogonality relation, we introduce an
approximation parameter into the admissibility condition and study the
resulting James-type extremal quantities as functions of that parameter.

The classical James constant, introduced by Gao and Lau
\cite{GaoLau1990}, is defined by
\[
J(X)
=
\sup\left\{
\min\{\norm{x+y},\norm{x-y}\}:
x,y\in S_X
\right\},
\]
where
$
S_X=\{x\in X:\norm{x}=1\}.
$

For every real normed space of dimension at least two,
$
\sqrt{2}\leq J(X)\leq2.
$
Moreover,
$
J(X)<2
$
if and only if $X$ is uniformly non-square.

It is worth emphasizing that the James constant is one of the fundamental
geometric quantities in the study of Banach spaces and has played an important
role in the quantitative analysis of the geometry of their unit balls.
Since its introduction by Gao and Lau, the James constant has been closely
related to uniform nonsquareness, normal structure, and metric fixed point
theory, and various inequalities connecting it with other geometric parameters
have been established
\cite{GaoLau1990,Kato2001,GarciaFalset2006,TakahashiKato2009}.
In particular, considerable attention has been devoted to the computation and
estimation of the James constant in concrete Banach spaces. Exact values and
sharp estimates have been obtained for Lorentz sequence spaces, Ces\`aro and
Ces\`aro--Orlicz sequence spaces, interpolation spaces, and several special
two-dimensional normed spaces
\cite{MitaniSaitoSuzuki2008,MaligrandaPetrotSuantai2007,
	BetiukPilarskaPhothiPrus2011,KomuroSaitoTanaka2016}.
The James constant has also been investigated from the viewpoint of the
geometry of normed planes, including Radon planes, as well as through its
relationship with isosceles orthogonality and extremal configurations on the
unit sphere
\cite{Mizuguchi2020,SainGhoshPaul2023}.
Moreover, a number of generalized James constants and related variants have
been introduced in order to capture finer geometric information and to
establish further connections with convexity, smoothness, and other geometric
constants
\cite{DhompongsaKaewkhaoTasena2003,MartiniPapiniWu2024,XiaoZhu2025}.
These developments illustrate the usefulness of the James constant as a
quantitative tool for describing nonsquareness and the geometric structure of
normed spaces.

Motivated by these developments, we consider James-type extremal quantities
under approximate isosceles and approximate Birkhoff--James orthogonality
constraints. This approach allows the classical James geometry to be compared
with the geometry selected by approximate orthogonality conditions.

The paper is organized into four sections. Section~\ref{sec:foundations}
introduces the approximate James-type constants and establishes their
basic estimates and geometric consequences. Section~\ref{sec:profile}
develops the structural theory of the Birkhoff profile, including
continuity, saturation, the saturation parameter, normal-structure
criteria, and Banach--Mazur stability. Section~\ref{sec:examples} presents explicit examples and exact
computations, including classical spaces, a mixed normed plane with a
nonconstant profile, and examples showing that neither the James
constant nor the endpoint constants together with the saturation
parameter determine the full Birkhoff profile.

\section{Preliminaries and Basic Properties}\label{sec:foundations}

Throughout the paper $X$ denotes a real normed linear space with norm $\norm{\cdot}$.  Its unit sphere and closed unit ball are denoted by $S_X$ and $B_X$.  Unless explicitly stated otherwise, the dimension of $X$ is at least two.

\medskip
The basic notation in place, we now introduce the two approximate orthogonality relations and the associated endpoint functionals.  The definitions are followed by a discussion of why the endpoint formulation is needed for a James-type quantity.

Motivated by the approximate isosceles orthogonality studied in
\cite{ChmielinskiWojcik2010}, we adopt the following multiplicative convention.

Motivated by the James constant, we study approximate orthogonality from the viewpoint of geometric constants. By restricting the classical extremal quantity to approximately orthogonal pairs, we obtain parameterized James-type constants that quantify how approximate orthogonality interacts with the geometry of the unit sphere and provide a natural tool for comparing different notions of approximate orthogonality.
\begin{definition}\label{def:endpoint-james}
	Let $0\leq\eps<1$. Define
	\[
	\JI(X)=\sup\{\min\{\norm{x+y},\norm{x-y}\}:x,y\in S_X,\ x\perp_{I,\eps}y\}
	\]
	and
	\[
	\JB(X)=\sup\{\min\{\norm{x+y},\norm{x-y}\}:x,y\in S_X,\ x\perp_{B,\eps}y\}.
	\]
	If the admissible family is empty, the corresponding supremum is interpreted as zero; in the spaces considered below the admissible families are non-empty.
\end{definition}

At the exact endpoint $\eps=0$, the Birkhoff--James endpoint constant
introduced above reduces exactly to the orthogonal James constant of
Baronti and Papini \cite{BarontiPapini2022}. Indeed,
\[
\begin{aligned}
	J_B^0(X)
	&=
	\sup\left\{
	\min\{\norm{x+y},\norm{x-y}\}:
	x,y\in S_X,\ x\perp_{B,0}y
	\right\} \\
	&=
	\sup\left\{
	\min\{\norm{x+y},\norm{x-y}\}:
	x,y\in S_X,\ x\perp_B y
	\right\} \\
	&=
	J_{\perp}(X).
\end{aligned}
\]
Therefore,
\[
J_{\perp}(X)=J_B^0(X)
\leq J_B^\eps(X)
\leq J(X),
\qquad 0\leq\eps<1,
\]
and
\[
\lim_{\eps\to1^-}J_B^\eps(X)=J(X).
\]
Hence, the family $\{J_B^\eps(X)\}_{0\leq\eps<1}$ contains
$J_{\perp}(X)$ as the special case $\eps=0$ and extends the orthogonal
James constant from exact Birkhoff--James orthogonality to its
approximate counterpart, while recovering the classical James constant
in the limit as $\eps\to1^-$. 

The endpoint formulation is chosen to preserve the James-constant scale and to detect square-like behavior.

\begin{lemma}\label{lem:exist-exact}
	Let $X$ be a real normed space with $\dim X\geq2$.
	\begin{itemize}
		\item[(i)] There exist $x,y\in S_X$ with $x\perp_B y$.
		\item[(ii)] There exist $u,v\in S_X$ with $u\perp_I v$.
	\end{itemize}
\end{lemma}

The assertion is standard; see, for example, \cite{James1947,AlonsoMartiniWu2012}.

\medskip
We next record several basic properties.

\begin{proposition}
	Let $X$ be a real normed space with $\dim X\geq2$ and let
	$0\leq\eps<1$. Then
	\[
	1\leq J_B^\eps(X)\leq J(X)\leq2.
	\]
\end{proposition}

\begin{proof}
	The upper bound follows immediately from the definition. For the lower
	bound, choose $x,y\in S_X$ with $x\perp_B y$ by
	Lemma~\ref{lem:exist-exact}(i). Then $x\perp_{B,\eps}y$ for every
	$\eps\in[0,1)$ and
	$\norm{x+y}\geq1$, $\norm{x-y}\geq1$.
	Hence $J_B^\eps(X)\geq1$.
\end{proof}

\begin{proposition}
	If $0\leq\eps_1\leq\eps_2<1$, then
	\[
	J_B^{\eps_1}(X)\leq J_B^{\eps_2}(X).
	\]
\end{proposition}

\begin{proof}
	If $x\perp_{B,\eps_1}y$, then
	\[
	\norm{x+\lambda y}
	\geq (1-\eps_1)\norm{x}
	\geq (1-\eps_2)\norm{x}
	\qquad(\lambda\in\mathbb R).
	\]
	Thus every $B$-$\eps_1$-admissible pair is also
	$B$-$\eps_2$-admissible, and the result follows.
\end{proof}

\begin{proposition}
	Let $x,y\in S_X$ and suppose $x\perp_{I,\eps}y$.  Then
	\[
	\min\{\norm{x+y},\norm{x-y}\}\geq \frac{2(1-\eps)}{2-\eps}.
	\]
	In particular, every admissible pair in the definition of $\JI(X)$ has endpoint minimum bounded away from zero for fixed $\eps<1$.
\end{proposition}

\begin{proof}
	Put $a=\norm{x+y}$ and $b=\norm{x-y}$.  Since $2=\norm{2x}\leq a+b$, at least one endpoint is non-zero.  If $a\leq b$, then the condition $a\geq(1-\eps)b$ gives $b\leq a/(1-\eps)$, and hence
	$
	2\leq a+b\leq a+\frac{a}{1-\eps}
	=a\frac{2-\eps}{1-\eps}.
	$
	Thus $\min\{\norm{x+y},\norm{x-y}\}=a\geq 2(1-\eps)/(2-\eps)$.  If $b\leq a$, then $a\leq(1+\eps)b$, so
	$
	2\leq a+b\leq(2+\eps)b,
	$
	and $\min\{\norm{x+y},\norm{x-y}\}=b\geq2/(2+\eps)$.  Since
	$
	\frac{2}{2+\eps}\geq \frac{2(1-\eps)}{2-\eps}
	\quad(0\leq\eps<1),
	$
	the asserted estimate follows.
\end{proof}

\begin{proposition}\label{pop:recovery}
	Let $X$ be a real normed space with $\dim X\geq2$.  Then
	\[
	\lim_{\eps\to1-}\JI(X)=J(X),
	\qquad
	\lim_{\eps\to1-}\JB(X)=J(X).
	\]
	Consequently
	\[
	\sup_{0\leq\eps<1}\JI(X)=\sup_{0\leq\eps<1}\JB(X)=J(X).
	\]
\end{proposition}

\begin{proof}
	The upper bounds $\JI(X)\leq J(X)$ and $\JB(X)\leq J(X)$ are already known.  It remains to prove the reverse limiting inequalities.
	
	First consider $\JI(X)$.  If $J(X)=1$, then $\JI(X)\geq1=J(X)$ for every $\eps$, and there is nothing to prove.  Assume $J(X)>1$.  Let $r$ be any number with $1<r<J(X)$.  Choose $x,y\in S_X$ such that
	$
	\min\{\norm{x+y},\norm{x-y}\}>r.
	$
	Put $a=\norm{x+y}$ and $b=\norm{x-y}$.  Then $r<a,b\leq2$.  Hence
	$
	\frac{a}{b}\leq \frac2r<2,
	\frac{a}{b}\geq \frac r2>0.
	$
	Choose $\eps_r<1$ so large that
	$
	1-\eps_r\leq \frac r2
	\hbox{and}
	1+\eps_r\geq \frac2r.
	$
	Then for every $\eps\in[\eps_r,1)$ the pair $(x,y)$ is $I$-$\eps$-orthogonal.  Thus $\JI(X)\geq \min\{\norm{x+y},\norm{x-y}\}>r$ for all such $\eps$.  Letting $r\uparrow J(X)$ proves the first limiting formula.
	
	Now consider $\JB(X)$.  Again there is nothing to prove if $J(X)=1$.  Let $1<r<J(X)$ and choose $x,y\in S_X$ with $\min\{\norm{x+y},\norm{x-y}\}>r$.  Since both endpoints have norm greater than one, $x$ and $y$ cannot be linearly dependent.  Therefore the one-dimensional subspace $\R y$ is a proper closed subspace of the two-dimensional space $\operatorname{span}\{x,y\}$, and
	$
	d=\dist(x,\R y)=\inf_{\lambda\in\R}\norm{x+\lambda y}>0.
	$
	If $\eps>1-d$, then
	$
	\norm{x+\lambda y}\geq d\geq1-\eps
	\quad(\lambda\in\R),
	$
	so $x\perp_{B,\eps}y$.  Hence $\JB(X)\geq \min\{\norm{x+y},\norm{x-y}\}>r$ for all sufficiently large $\eps<1$.  Letting $r\uparrow J(X)$ completes the proof.Here and throughout the paper, $r\uparrow J(X)$ means that
	$r$ approaches $J(X)$ from below.
\end{proof}

\begin{proposition}
	Let $X$ be a real normed space with $\dim X\geq2$ and let $0\leq\eps<1$.  Then
	\[
	\JI(X)^2\leq 2\CNJ(X),
	\qquad
	\JB(X)^2\leq 2\CNJ(X),
	\]
	and
	\[
	\JI(X)^2\leq 2\CZ(X),
	\qquad
	\JB(X)^2\leq 2\CZ(X).
	\]
	Equivalently,
	\[
	\JI(X),\JB(X)\leq \min\{\sqrt{2\CNJ(X)},\sqrt{2\CZ(X)},J(X)\}.
	\]
\end{proposition}

\begin{proof}
	We prove the estimates for a general admissible pair; the result follows by taking suprema.  Let $x,y\in S_X$ and put $m=\min\{\norm{x+y},\norm{x-y}\}$.  Since $\norm{x}=\norm{y}=1$,
	$
	\CNJ(X)\geq \frac{\norm{x+y}^2+\norm{x-y}^2}{4}
	\geq \frac{2m^2}{4}=\frac{m^2}{2}.
	$
	Similarly,
	$
	\CZ(X)\geq \frac{\norm{x+y}\norm{x-y}}2\geq \frac{m^2}{2}.
	$
	Taking the supremum over $I$-$\eps$ admissible pairs gives the estimates for $\JI(X)$, and taking the supremum over $B$-$\eps$ admissible pairs gives the estimates for $\JB(X)$.
\end{proof}

\begin{corollary}
	For each $0\leq\eps<1$,
	\[
	\CNJ(X)\geq \frac12\max\{\JI(X)^2,\JB(X)^2\},
	\qquad
	\CZ(X)\geq \frac12\max\{\JI(X)^2,\JB(X)^2\}.
	\]
	In particular, if either approximate endpoint constant is close to $2$, then both $\CNJ(X)$ and $\CZ(X)$ are close to their maximal value $2$.
\end{corollary}

\begin{theorem}
	Let $A$ denote either $I$ or $B$, and write $J_A^\eps(X)$ for the corresponding approximate endpoint constant.  If $0<t<J_A^\eps(X)$, then
	$
	\delta_X(t)\leq 1-\frac t2.
	$
	Consequently, if for some $t\in(0,2)$ one has
	$
	\delta_X(t)>1-\frac t2,
	$
	then
	\[
	J_I^\eps(X)\leq t
	\quad\hbox{and}\quad
	J_B^\eps(X)\leq t
	\qquad(0\leq\eps<1).
	\]
\end{theorem}

\begin{proof}
	Assume $0<t<J_A^\eps(X)$.  By the definition of the supremum there exist admissible $x,y\in S_X$ such that $\min\{\norm{x+y},\norm{x-y}\}>t$.  In particular,
	$
	\norm{x-y}>t,
	\norm{x+y}>t.
	$
	Since $\norm{x-y}\geq t$, the definition of $\delta_X(t)$ gives
	$
	\delta_X(t)
	\leq 1-\frac{\norm{x+y}}2
	<1-\frac t2.
	$
	This proves the first assertion.  The final assertion is the contrapositive.
\end{proof}

\begin{corollary}
	Suppose there exists $t_0\in(0,2)$ such that $\delta_X(t_0)>1-t_0/2$.  Then all endpoint approximate James constants are bounded by $t_0$:
	\[
	\sup_{0\leq\eps<1}J_I^\eps(X)\leq t_0,
	\qquad
	\sup_{0\leq\eps<1}J_B^\eps(X)\leq t_0.
	\]
	In particular $J(X)\leq t_0<2$.
\end{corollary}

\begin{proof}
	The preceding theorem gives the first two estimates for every fixed $\eps$.  Taking suprema over $\eps$ and using Proposition~\ref{pop:recovery} yields $J(X)\leq t_0$.
\end{proof}

\medskip
These estimates lead naturally to uniform non-squareness and its metric fixed point consequences.  We record the relevant equivalences before turning to concrete computations.

A Banach space $X$ is uniformly non-square\cite{James1964} if there exists $\eta>0$ such that for all $x,y\in S_X$,
\[
\min\{\norm{x+y},\norm{x-y}\}\leq 2-\eta.
\]
Equivalently, $J(X)<2$.  This equivalence is a classical theorem of James.

\begin{theorem}\label{thm:uniform-nonsquare}
	Let $X$ be a real Banach space with $\dim X\geq2$.  The following are equivalent.
	\begin{itemize}
		\item[(i)] $X$ is uniformly non-square.
		\item[(ii)] $\displaystyle \sup_{0\leq\eps<1}J_I^\eps(X)<2$.
		\item[(iii)] $\displaystyle \sup_{0\leq\eps<1}J_B^\eps(X)<2$.
	\end{itemize}
\end{theorem}

\begin{proof}
	By Proposition~\ref{pop:recovery},
	$
	\displaystyle\sup_{0\leq\eps<1}J_I^\eps(X)
	=\displaystyle\sup_{0\leq\eps<1}J_B^\eps(X)=J(X).
	$
	Thus (ii) and (iii) are both equivalent to $J(X)<2$, which is equivalent to uniform non-squareness.
\end{proof}

\begin{corollary}
	Let $X$ be a real Banach space with $\dim X\geq2$. If
	\[
	\sup_{0\leq\eps<1}J_I^\eps(X)<2
	\quad\hbox{or}\quad
	\sup_{0\leq\eps<1}J_B^\eps(X)<2,
	\]
	then $X$ has the fixed point property for nonexpansive mappings.
\end{corollary}

\begin{proof}
	By Theorem~\ref{thm:uniform-nonsquare}, either assumption implies
	$J(X)<2$, and hence $X$ is uniformly non-square. The conclusion
	therefore follows from the fixed point theorem for uniformly
	non-square Banach spaces in \cite{GarciaFalset2006}.
\end{proof}

\section{Approximate James constants and their geometric properties}\label{sec:profile}

The original definitions admit a substantial simplification on the isosceles side.  We first record this fact and then focus on the genuinely varying Birkhoff profile.  Throughout this section, set
\[
d(x,y)=\operatorname{dist}(x,\mathbb{R}y),\qquad x,y\in S_X.
\]
The condition $x\perp_{B,\eps}y$ is equivalent to $d(x,y)\geq1-\eps$ and is the global approximate Birkhoff--James condition associated with Dragomir's formulation \cite{Dragomir1991}.

\begin{proposition}\label{prop:isosc-collapse}
	For every real normed space $X$ with $\dim X\geq2$ and every $0\leq\eps<1$,
	\[
	J_I^\eps(X)=J(X).
	\]
	Consequently, only the Birkhoff profile
	$\eps\mapsto J_B^\eps(X)$ may vary with $\eps$.
\end{proposition}

\begin{proof}
	Fix $0\leq\eps<1$. Recall that
	\[
	J_I^\eps(X)
	=
	\sup\left\{
	\min\{\norm{x+y},\norm{x-y}\}:
	x,y\in S_X,\ x\perp_{I,\eps}y
	\right\},
	\]
	where $x\perp_{I,\eps}y$ means that
	\[
	(1-\eps)\norm{x-y}
	\leq \norm{x+y}
	\leq (1+\eps)\norm{x-y}.
	\]
	We prove the equality by establishing both inequalities.
	
	First, let $x,y\in S_X$ satisfy $x\perp_{I,\eps}y$.
	The definition of the classical James constant gives
	\[
	\min\{\norm{x+y},\norm{x-y}\}
	\leq
	\sup_{u,v\in S_X}
	\min\{\norm{u+v},\norm{u-v}\}
	=J(X).
	\]
	This estimate holds for every pair admissible in the
	definition of $J_I^\eps(X)$. Taking the supremum over
	all such pairs, we obtain	$	J_I^\eps(X)\leq J(X).$
	For the reverse inequality, we use the classical
	isosceles representation of the James constant:
	\[
	J(X)
	=
	\sup\left\{
	\norm{x+y}:
	x,y\in S_X,\ \norm{x+y}=\norm{x-y}
	\right\};
	\]
	see, for example, \cite{Kato2001,BarontiPapini2022}.
	Let $x,y\in S_X$ satisfy
	$\norm{x+y}=\norm{x-y}$, and write
	\[
	a=\norm{x+y}=\norm{x-y}.
	\]
	Since $a\geq0$ and $0\leq\eps<1$, we have
	\[
	(1-\eps)a\leq a\leq(1+\eps)a.
	\]
	Equivalently,
	\[
	(1-\eps)\norm{x-y}
	\leq \norm{x+y}
	\leq (1+\eps)\norm{x-y}.
	\]
	Thus $x\perp_{I,\eps}y$, and this exact isosceles pair
	is admissible in the definition of $J_I^\eps(X)$.
	Consequently,
	\[
	\norm{x+y}
	=
	\min\{\norm{x+y},\norm{x-y}\}
	\leq J_I^\eps(X).
	\]
	Taking the supremum over all exact isosceles pairs
	in $S_X\times S_X$ and applying the representation above,
	we conclude that$	J(X)\leq J_I^\eps(X).$
	Combining the two inequalities yields	$	J_I^\eps(X)=J(X).$
	Since $\eps$ was arbitrary and $J(X)$ is independent of
	$\eps$, the isosceles profile is constant on $[0,1)$.
	Hence, among the two profiles considered here, only
	the Birkhoff profile can be nonconstant.
\end{proof}

Write $J_\perp(X)=J_B^0(X)$ for the orthogonal James constant of Baronti and Papini \cite{BarontiPapini2022}.  Then
\[
1\leq J_\perp(X)\leq J_B^\eps(X)\leq J(X)\leq2.
\]
Moreover, the estimate $J_\perp(X)\geq2J(X)-2$ from \cite{BarontiPapini2022} shows that, for each fixed $\eps$,
\[
J_B^\eps(X)<2\quad\Longleftrightarrow\quad J(X)<2.
\]

The next two lemmas isolate the geometric mechanisms underlying the
Birkhoff profile. The first gives a dual characterization and a
two-dimensional reduction, while the second converts large endpoint
values into a quantitative saturation bound.
\begin{lemma}\label{lem:annihilator-distance}
	For $x,y\in S_X$, the minimum defining $d(x,y)$ is attained and
	\[
	d(x,y)=\max\{f(x): f\in B_{X^*},\ f(y)=0\}.
	\]
	If $d(x,y)>0$, the maximum can be taken over $S_{X^*}$.  In addition,
	\[
	J_B^\eps(X)=\sup\{J_B^\eps(E):E\subseteq X,\ \dim E=2\}.
	\]
\end{lemma}

\begin{proof}
	The function $t\mapsto\norm{x+ty}$ is continuous and bounded below by $|t|-1$, hence it attains its minimum.  Every $f\in B_{X^*}$ with $f(y)=0$ satisfies $f(x)\leq d(x,y)$.  If $x,y$ are independent, define on $\operatorname{span}\{x,y\}$
	$
	f_0(ax+by)=a\,d(x,y).
	$
	The definition of the distance gives $|a|d(x,y)\leq\norm{ax+by}$.  At a minimizing point $x+t_0y$ equality is attained, hence $\norm{f_0}=1$; Hahn--Banach extends $f_0$ to $X$.  The dependent case is immediate.  Finally, every $B$-$\eps$ admissible pair is independent because $1-\eps>0$, and its endpoint norms and its distance to the generated line are unchanged inside its two-dimensional span.
\end{proof}

\begin{lemma}\label{lem:saturation-bound}
	If $x,y\in S_X$ and $\min\{\norm{x+y},\norm{x-y}\}>1$, then
	\[
	d(x,y)\geq\frac{\min\{\norm{x+y},\norm{x-y}\}}{2}.
	\]
	Consequently,
	\[
	J_B^\eps(X)=J(X)
	\qquad\text{whenever}\qquad
	\eps>1-\frac{J(X)}2.
	\]
	Since $J(X)\geq\sqrt2$, in particular,
	\[
	J_B^\eps(X)=J(X)
	\qquad
	\text{for every }
	\eps\in\left(1-\frac1{\sqrt2},1\right).
	\]
\end{lemma}

\begin{proof}
	Let
	$
	F(t)=\norm{x+ty},
	\qquad t\in\mathbb{R}.
	$
	Since $F$ is continuous and
	$
	F(t)\geq |t|\norm{y}-\norm{x}=|t|-1,
	$
	we have $F(t)\to\infty$ as $|t|\to\infty$. Hence $F$ attains its
	minimum at some $t_0\in\mathbb{R}$, and
	$
	d(x,y)=F(t_0).
	$
	
	Because
	$
	F(0)=\norm{x}=1
	$
	and
	$
	F(1)=\norm{x+y}>1,
	F(-1)=\norm{x-y}>1,
	$
	the convexity of $F$ implies that every minimizing point lies in
	$(-1,1)$. Replacing $y$ by $-y$ if necessary, we may assume
	$
	0\leq t_0<1.
	$
	Set
	$
	a=F(t_0)=d(x,y).
	$
	
	Since $\norm{y}=1$, the reverse triangle inequality gives
	\[
	|F(s)-F(t)|
	\leq \norm{(s-t)y}
	=|s-t|,
	\qquad s,t\in\mathbb{R}.
	\]
	Thus $F$ is $1$-Lipschitz. Applying this first with $s=0$ and
	$t=t_0$, we obtain
	$
	1-a
	=
	F(0)-F(t_0)
	\leq t_0.
	$
	Moreover, since
	$
	\min\{\norm{x+y},\norm{x-y}\}
	=
	\min\{F(1),F(-1)\}
	\leq F(1),
	$
	we have
	$
	\min\{\norm{x+y},\norm{x-y}\}-a
	\leq
	F(1)-F(t_0)
	\leq
	1-t_0.
	$
	Adding the last two inequalities yields
	$
	1+\min\{\norm{x+y},\norm{x-y}\}-2a\leq1,
	$
	and therefore
	$
	a\geq\frac{\min\{\norm{x+y},\norm{x-y}\}}2.
	$
	Since $a=d(x,y)$, this proves
	$
	d(x,y)\geq\frac{\min\{\norm{x+y},\norm{x-y}\}}2.
	$
	
	We now prove the saturation statement. Let
	$
	\eps>1-\frac{J(X)}2.
	$
	Then
	$
	2(1-\eps)<J(X).
	$
	Choose any number $r$ such that
	$
	\max\{1,2(1-\eps)\}<r<J(X).
	$
	By the definition of $J(X)$, there exist $x,y\in S_X$ such that
	$
	\min\{\norm{x+y},\norm{x-y}\}>r.
	$
	Since $r>1$, the first part of the lemma applies and gives
	$
	d(x,y)
	\geq
	\frac{\min\{\norm{x+y},\norm{x-y}\}}2
	>
	\frac r2
	>
	1-\eps.
	$
	Hence
	$
	x\perp_{B,\eps}y,
	$
	so the pair $(x,y)$ is admissible in the definition of
	$J_B^\eps(X)$. Consequently,
	$
	J_B^\eps(X)\geq \min\{\norm{x+y},\norm{x-y}\}>r.
	$
	Since $r$ may be chosen arbitrarily close to $J(X)$ from below,
	$
	J_B^\eps(X)\geq J(X).
	$
	The reverse inequality
	$
	J_B^\eps(X)\leq J(X)
	$
	holds by definition. Therefore
	$
	J_B^\eps(X)=J(X).
	$
	
	Finally, since $J(X)\geq\sqrt2$,
	$
	1-\frac{J(X)}2
	\leq
	1-\frac1{\sqrt2}.
	$
	Hence every
	$
	\eps\in\left(1-\frac1{\sqrt2},1\right)
	$
	satisfies
	$
	\eps>1-\frac{J(X)}2,
	$
	and therefore
	$
	J_B^\eps(X)=J(X).
	$
\end{proof}

\begin{remark}\label{rem:cnj-comparison}
	The saturation result also gives a comparison with the
	von Neumann--Jordan constant. It is known that
	$
	C_{\rm NJ}(X)\leq J(X);
	$
	see, for example, \cite{Wang2010,TakahashiKato2009,YangLi2010,KatoMaligrandaTakahashi2010}. Hence Lemma~\ref{lem:saturation-bound}
	implies
	$
	C_{\rm NJ}(X)
	\leq J_B^\eps(X)
	=J(X)
	\left(
	\eps>1-\frac{J(X)}2
	\right).
	$

\end{remark}

\begin{lemma}\label{lem:correction-zero}
	Let $0<\eps<1/2$ and let $x,y\in S_X$ satisfy $d(x,y)\geq1-\eps$.  For every $r$ with $\sqrt{2\eps}<r<1$, there exist $u,v\in S_X$ such that $u\perp_Bv$ and
	\[
	\norm{x-u}<r,\qquad \norm{y-v}<2r.
	\]
	In particular,
	\[
	0\leq J_B^\eps(X)-J_B^0(X)\leq3\sqrt{2\eps}.
	\]
\end{lemma}

\begin{proof}
	Set $E=\operatorname{span}\{x,y\}$.  By Lemma~\ref{lem:annihilator-distance}, there is $f\in S_{E^*}$ with $f(y)=0$ and $f(x)\geq1-\eps$.  Choose $\delta$ with $\eps<\delta<r^2/2$.  The quantitative Bishop--Phelps--Bollob\'as theorem \cite[Theorem~2.1]{Chica2014} provides $u\in S_E$ and $g\in S_{E^*}$ such that
	\[
	g(u)=1,\qquad \norm{u-x}<r,\qquad \norm{g-f}<r.
	\]
	Let $w=y-g(y)u$.  Then $g(w)=0$, $\norm{w-y}=|g(y)|<r$ and $\norm{w}>1-r>0$.  With $v=w/\norm{w}$,
	$
	\norm{v-y}\leq\norm{v-w}+\norm{w-y}
	=|1-\norm{w}|+\norm{w-y}<2r.
	$
	Moreover, $\norm{u+tv}\geq g(u+tv)=1$ for all $t\in\mathbb{R}$, hence $u\perp_Bv$.  Each endpoint norm changes by less than $3r$, so $\min\{\norm{x+y},\norm{x-y}\}\leq J_B^0(X)+3r$.  Take suprema and then let $r\downarrow\sqrt{2\eps}$.
\end{proof}

\begin{lemma}\label{lem:correction-positive}
	Let $0<e<\eta<1$.  Then
	\[
	0\leq J_B^\eta(X)-J_B^e(X)
	\leq \frac{3(\eta-e)}{e(1-e)}.
	\]
	If $e<1/2$, the right-hand side can be replaced by $6(\eta-e)/e$.
\end{lemma}

\begin{proof}
	Take an $\eta$-admissible pair $x,y\in S_X$ and write $a=d(x,y)$ and $b=1-e$.  If $a\geq b$, no correction is needed.  Otherwise $1-\eta\leq a<b<1$.  Choose $t_0$ with $\norm{x+t_0y}=a$; the reverse triangle inequality gives $|t_0|\leq1+a\leq2$.  Put
	$
	s=\frac{1-a/b}{1-a}=\frac{b-a}{b(1-a)}\in(0,1).
	$
	Convexity gives $\norm{x+st_0y}\leq1-s+sa=a/b$.  By continuity there is $s_0\in[0,s]$ with $\norm{x+s_0t_0y}=a/b$.  Set $w=x+s_0t_0y$ and $z=(b/a)w$.  Translation along $\mathbb{R}y$ leaves the distance unchanged, so $\norm{z}=1$ and $d(z,y)=b$.  Furthermore,
	\[
	\begin{aligned}
		\norm{z-x}
		&\leq \norm{z-w}+\norm{w-x}\\
		&\leq 1-\frac ab+2s
		=\frac{b-a}{b}\left(1+\frac{2}{1-a}\right)
		\leq\frac{3(\eta-e)}{e(1-e)}.
	\end{aligned}
	\]
	Here $b-a\leq\eta-e$ and $1-a>e$.  Each endpoint norm changes by at most $\norm{z-x}$, and taking suprema gives the result.
\end{proof}

\begin{theorem}\label{thm:profile-continuity}
	For every real normed space $X$ with $\dim X\geq2$, the function
	\[
	\eps\longmapsto J_B^\eps(X)
	\]
	is continuous on $[0,1)$ and locally Lipschitz on $(0,1)$.  Uniformly in $X$, for $0\leq e\leq\eta<1$,
	\[
	0\leq J_B^\eta(X)-J_B^e(X)
	\leq\min\{1,6(\eta-e)^{1/3}\}.
	\]
	At zero the sharper estimate of Lemma~\ref{lem:correction-zero} holds.  The profile extends continuously to $\eps=1$ by the value $J(X)$, and
	\[
	J_B^\eps(X)=J(X)
	\qquad\left(1-\frac{J(X)}2\leq\eps\leq1\right).
	\]
\end{theorem}

\begin{proof}
	Monotonicity and Lemma~\ref{lem:correction-zero} give right continuity at zero, while Lemma~\ref{lem:correction-positive} gives local Lipschitz continuity at every positive parameter.  The closed saturation threshold follows from Lemma~\ref{lem:saturation-bound} by continuity when $J(X)<2$.  If $J(X)=2$, the inequality $J_\perp(X)\geq2J(X)-2$ gives $J_B^0(X)=2$.
	
	It remains to prove the common modulus.  Put $h=\eta-e>0$.  Since $1\leq J_B^\eps(X)\leq2$, the difference is at most $1$.  If $h\geq1/8$ there is nothing further to prove.  Assume $h<1/8$.  If $e\geq1/2$, both parameters lie in the constant part of the profile.  If $h^{2/3}<e<1/2$, Lemma~\ref{lem:correction-positive} gives
	$
	J_B^\eta(X)-J_B^e(X)\leq\frac{6h}{e}<6h^{1/3}.
	$
	Finally, if $e\leq h^{2/3}$, then $\eta\leq h^{2/3}+h<\tfrac32h^{2/3}<1/2$, and Lemma~\ref{lem:correction-zero} together with monotonicity yields
	$
	J_B^\eta(X)-J_B^e(X)
	\leq J_B^\eta(X)-J_B^0(X)
	\leq3\sqrt{2\eta}<6h^{1/3}.
	$
	This proves the uniform estimate and the extension to $\eps=1$.
\end{proof}

Although approximate Birkhoff--James orthogonality is originally
defined for $0\leq\eps<1$, the endpoint $\eps=1$ admits a natural
interpretation. Indeed, when $\eps=1$,
\[
\norm{x+\lambda y}
\geq
(1-\eps)\norm{x}
=
0
\qquad
(\lambda\in\mathbb R),
\]
so every pair $x,y\in S_X$ is admissible. Hence
$
J_B^1(X)
=
\sup_{x,y\in S_X} \min\{\norm{x+y},\norm{x-y}\}
=
J(X).
$
On the other hand, Theorem~\ref{thm:profile-continuity} gives
$
\displaystyle\lim_{\eps\to1^-}J_B^\eps(X)=J(X).
$
Thus the assignment
$
J_B^1(X)=J(X)
$
is both the natural endpoint value of the defining extremal problem
and the continuous extension of the Birkhoff profile to $[0,1]$.

Since the Birkhoff profile is nondecreasing and extends continuously to
the endpoint value $J_B^1(X)=J(X)$, it is natural to ask when it first
reaches the classical James constant. This leads to the following
threshold parameter.
\begin{definition}\label{def:birkhoff-saturation}
	The \emph{Birkhoff saturation parameter} of $X$ is defined by
	\[
	\sigma_B(X)
	=
	\inf\left\{
	\eps\in[0,1]:
	J_B^\eps(X)=J(X)
	\right\}.
	\]
\end{definition}

For $1<r<J(X)$, put
\[
D_X(r)
=
\sup\left\{
d(x,y):
x,y\in S_X,\ \min\{\norm{x+y},\norm{x-y}\}>r
\right\}.
\]
Since the sets in the above supremum decrease as $r$ increases,
the function $D_X$ is nonincreasing and therefore
$\displaystyle\lim_{r\uparrow J(X)}D_X(r)$ exists.

\begin{theorem}\label{thm:birkhoff-saturation-parameter}
	Let $X$ be a real normed space with $\dim X\geq2$. Then
	\[
	\left\{
	\eps\in[0,1]:
	J_B^\eps(X)=J(X)
	\right\}
	=
	[\sigma_B(X),1].
	\]
	Moreover,
	\[
	\sigma_B(X)
	=
	1-
	\lim_{r\uparrow J(X)}
	\sup\left\{
	d(x,y):
	x,y\in S_X,\ \min\{\norm{x+y},\norm{x-y}\}>r
	\right\}.
	\]
	The saturation parameter satisfies
	\[
	\frac{\bigl(J(X)-J_\perp(X)\bigr)^2}{18}
	\leq
	\sigma_B(X)
	\leq
	1-\frac{J(X)}2
	\leq
	1-\frac1{\sqrt2}.
	\]
	If $X$ is finite dimensional, then
	\[
	\sigma_B(X)
	=
	1-
	\max\left\{
	d(x,y):
	x,y\in S_X,\ \min\{\norm{x+y},\norm{x-y}\}=J(X)
	\right\}.
	\]
\end{theorem}

\begin{proof}
	By Proposition~\ref{prop:isosc-collapse} and
	Theorem~\ref{thm:profile-continuity}, the function
	$
	\eps\mapsto J_B^\eps(X)
	$
	is continuous and nondecreasing on $[0,1]$, where
	$J_B^1(X)=J(X)$. Hence
	$
	\left\{
	\eps\in[0,1]:
	J_B^\eps(X)=J(X)
	\right\}
	$
	is a nonempty closed upper interval. Therefore it is equal to
	$[\sigma_B(X),1]$.
	
	Set
	$
	D_*=
	\displaystyle\lim_{r\uparrow J(X)}D_X(r).
	$
	We claim that
	$
	\sigma_B(X)=1-D_*.
	$
	
	Let $\eps>1-D_*$. Then
	$
	1-\eps<D_*\leq D_X(r)
	\qquad
	(1<r<J(X)).
	$
	By the definition of $D_X(r)$, for every $r<J(X)$ there exist
	$x,y\in S_X$ such that
	\[
	\min\{\norm{x+y},\norm{x-y}\}>r
	\qquad\text{and}\qquad
	d(x,y)>1-\eps.
	\]
	Thus $x\perp_{B,\eps}y$, and hence
	$
	J_B^\eps(X)\geq r.
	$
	Since this holds for every $r<J(X)$,
	$
	J_B^\eps(X)=J(X).
	$
	It follows that
	$
	\sigma_B(X)\leq1-D_*.
	$
	
	Conversely, let $\eps<1-D_*$. Then
	$
	1-\eps>D_*.
	$
	Since $D_X(r)$ decreases to $D_*$ as $r\uparrow J(X)$, there exists
	$r_0<J(X)$ such that
	$
	D_X(r_0)<1-\eps.
	$
	Consequently, no pair $x,y\in S_X$ satisfying
	$d(x,y)\geq1-\eps$ can have $\min\{\norm{x+y},\norm{x-y}\}>r_0$. Therefore
	$
	J_B^\eps(X)\leq r_0<J(X),
	$
	and hence $\eps<\sigma_B(X)$. Thus
	$
	\sigma_B(X)\geq1-D_*,
	$
	which proves
	$
	\sigma_B(X)
	=
	1-
	\displaystyle\lim_{r\uparrow J(X)}D_X(r).
	$
	
	The upper estimate follows directly from
	Theorem~\ref{thm:profile-continuity}:
	$
	\sigma_B(X)
	\leq
	1-\frac{J(X)}2.
	$
	Since $J(X)\geq\sqrt2$,
	$
	\sigma_B(X)
	\leq
	1-\frac1{\sqrt2}
	<
	\frac12.
	$
	
	We next prove the lower estimate. If $\sigma_B(X)=0$, then
	$
	J_\perp(X)
	=
	J_B^0(X)
	=
	J(X),
	$
	and the assertion is immediate. Suppose therefore that
	$\sigma_B(X)>0$. Since
	$
	0<\sigma_B(X)<\frac12
	$
	and
	$
	J_B^{\sigma_B(X)}(X)=J(X),
	$
	Lemma~\ref{lem:correction-zero} gives
	$
	J(X)-J_\perp(X)
	=
	J_B^{\sigma_B(X)}(X)-J_B^0(X)
	\leq
	3\sqrt{2\sigma_B(X)}.
	$
	Squaring both sides yields
	$
	\sigma_B(X)
	\geq
	\frac{\bigl(J(X)-J_\perp(X)\bigr)^2}{18}.
	$
	
	Finally, assume that $X$ is finite dimensional. Set
	\[
	\mathcal E_J(X)
	=
	\left\{
	(x,y)\in S_X\times S_X:
	\min\{\norm{x+y},\norm{x-y}\}=J(X)
	\right\}.
	\]
	Since $S_X\times S_X$ is compact and $m$ is continuous,
	$\mathcal E_J(X)$ is nonempty and compact. Moreover, $d$ is continuous
	on $S_X\times S_X$. Indeed, for $x,y\in S_X$,
	$
	d(x,y)
	=
	\min_{|t|\leq2}\norm{x+ty},
	$
	because
	\[
	\norm{x+ty}\geq |t|-1>1
	\qquad
	\text{whenever } |t|>2,
	\]
	whereas $d(x,y)\leq\norm{x}=1$.
	Hence the maximum
	$
	D_J
	=
	\max\left\{
	d(x,y):
	(x,y)\in\mathcal E_J(X)
	\right\}
	$
	is attained.
	
	Choose $(x_0,y_0)\in\mathcal E_J(X)$ such that
	$d(x_0,y_0)=D_J$. If
	$
	\eps\geq1-D_J,
	$
	then
	$
	d(x_0,y_0)\geq1-\eps,
	$
	so $(x_0,y_0)$ is $B$-$\eps$-orthogonal. Since
	$
	m(x_0,y_0)=J(X),
	$
	we obtain
	$
	J_B^\eps(X)=J(X).
	$
	Therefore
	$
	\sigma_B(X)\leq1-D_J.
	$
	
	Conversely, suppose that $\eps<1-D_J$. If
	$J_B^\eps(X)=J(X)$, then the admissible set
	$
	\left\{
	(x,y)\in S_X\times S_X:
	d(x,y)\geq1-\eps
	\right\}
	$
	is compact, so the supremum defining $J_B^\eps(X)$ is attained.
	Thus there exist $x,y\in S_X$ such that
	\[
	\min\{\norm{x+y},\norm{x-y}\}=J(X)
	\qquad\text{and}\qquad
	d(x,y)\geq1-\eps>D_J,
	\]
	contradicting the definition of $D_J$. Hence
	$
	J_B^\eps(X)<J(X),
	$
	and therefore
	$
	\sigma_B(X)\geq1-D_J.
	$
	Consequently,
	$
	\sigma_B(X)
	=
	1-
	\max\left\{
	d(x,y):
	x,y\in S_X,\ \min\{\norm{x+y},\norm{x-y}\}=J(X)
	\right\}.
	$
\end{proof}

We first recall the notion of normal structure. A bounded convex subset
$C$ of a Banach space $X$ is said to have \emph{normal structure} if
every bounded convex subset $H\subseteq C$ with more than one point
contains a point $x_0\in H$ such that
$
\sup_{y\in H}\norm{x_0-y}<\diam(H).
$
The space $X$ is said to have \emph{uniform normal structure} if there
exists a constant $\lambda<1$ such that, for every nonempty bounded
closed convex subset $C\subseteq X$ with $\diam(C)>0$,
$
r(C)\leq\lambda\,\diam(C),
$
where
$
r(C)=\inf_{x\in C}\sup_{y\in C}\norm{x-y}.
$
Normal structure was introduced by Brodski\u{\i} and Milman
\cite{BrodskiiMilman1948}; uniform normal structure is a quantitative
strengthening that plays an important role in the fixed point theory of
nonexpansive mappings; see, for example, \cite{Kato2001}.

The preceding saturation result also yields sufficient conditions for
uniform normal structure in terms of the two approximate endpoint
constants. Recall that a theorem of Dhompongsa, Kaewkhao and Tasena
\cite{DhompongsaKaewkhaoTasena2003} asserts that
$
J(X)<\frac{1+\sqrt5}{2}
$
implies that $X$ has uniform normal structure.

\begin{corollary}\label{cor:uniform-normal-structure}
	Let $X$ be a real Banach space with $\dim X\geq2$, and put
	\[
	\varphi=\frac{1+\sqrt5}{2},
	\qquad
	\varepsilon_c=\frac{3-\sqrt5}{4}.
	\]
	If $\varepsilon_c\leq\eps_0<1$, then
	\[
	J_B^{\eps_0}(X)<\varphi
	\quad\Longleftrightarrow\quad
	J(X)<\varphi.
	\]
	Consequently, $J_B^{\eps_0}(X)<\varphi$ implies that $X$ has uniform
	normal structure.
\end{corollary}

\begin{proof}
	The implication $J(X)<\varphi \Longrightarrow J_B^{\eps_0}(X)<\varphi$
	follows from $J_B^{\eps_0}(X)\leq J(X)$.
	
	Conversely, suppose that $J_B^{\eps_0}(X)<\varphi$ for some
	$\eps_0\geq\varepsilon_c$, and assume that $J(X)\geq\varphi$. Then
	\[
	1-\frac{J(X)}2
	\leq
	1-\frac{\varphi}{2}
	=
	\frac{3-\sqrt5}{4}
	=
	\varepsilon_c
	\leq\eps_0.
	\]
	Hence Theorem~\ref{thm:profile-continuity} yields
	$J_B^{\eps_0}(X)=J(X)\geq\varphi$, a contradiction. Therefore
	$J(X)<\varphi$.
	
	The final assertion follows from the James-constant criterion of
	Dhompongsa, Kaewkhao and Tasena.
\end{proof}

\medskip
Having established continuity with respect to the approximation parameter, we now study how the same profile behaves under perturbations of the ambient norm.  This leads to a parameter-transport principle and explicit Banach--Mazur stability estimates.

For isomorphic real normed spaces set
$
d_{BM}(X,Y)=\inf\{\norm{T}\norm{T^{-1}}:T:X\to Y\text{ is an isomorphism onto }Y\}.
$
Write
$
\omega(h)=\min\{1,6h^{1/3}\}.
$

\begin{theorem}\label{thm:BM-stability}
	Let $D\geq d_{BM}(X,Y)$ and $D\geq1$.  For $0\leq\eps<1$, put
	$
	\eta=1-\frac{1-\eps}{D}.
	$
	Then
	$
	J_B^\eps(X)+1\leq D\bigl(J_B^\eta(Y)+1\bigr),
	$
	and the same inequality holds with $X$ and $Y$ interchanged.  Consequently,
	\[
	\sup_{0\leq\eps<1}|J_B^\eps(X)-J_B^\eps(Y)|
	\leq 3(D-1)+D\,\omega(1-D^{-1}).
	\]
	For $1\leq D<2$,
	\[
	|J_\perp(X)-J_\perp(Y)|
	\leq3(D-1)+3\sqrt{2D(D-1)}.
	\]
	At each fixed $0<\eps<1/2$,
	\[
	|J_B^\eps(X)-J_B^\eps(Y)|
	\leq\left(3+\frac6\eps\right)(D-1),
	\]
	while for $1/2\leq\eps<1$ the right-hand side can be replaced by $3(D-1)$.
\end{theorem}

\begin{proof}
	First suppose that an isomorphism $T:X\to Y$ has distortion at most $D$.  After rescaling,
	\[
	\norm{z}_X\leq\norm{Tz}_Y\leq D\norm{z}_X\qquad(z\in X).
	\]
	For an $\eps$-admissible pair $x,y\in S_X$, set
	\[
	a=\norm{Tx}_Y,\quad b=\norm{Ty}_Y,\quad
	u=\frac{Tx}{a},\quad v=\frac{Ty}{b},\quad c=\max\{a,b\}.
	\]
	Then $1\leq a,b\leq D$ and
	\[
	\operatorname{dist}_Y(u,\mathbb{R}v)
	=\frac{\operatorname{dist}_Y(Tx,\mathbb{R}Ty)}{a}
	\geq\frac{d_X(x,y)}{D}
	\geq\frac{1-\eps}{D}=1-\eta.
	\]
	Thus $(u,v)$ is $\eta$-admissible in $Y$.  Normalizing both images introduces an endpoint error at most $|a-b|/c$.  For either sign,
	\[
	\begin{aligned}
		\norm{u\pm v}_Y
		&\geq\frac{\norm{T(x\pm y)}_Y-|a-b|}{c}\\
		&\geq\frac{\norm{x\pm y}_X-c+1}{c}
		=\frac{\norm{x\pm y}_X+1}{c}-1\\
		&\geq\frac{\norm{x\pm y}_X+1}{D}-1.
	\end{aligned}
	\]
	Taking endpoint minima and suprema gives
	$
	J_B^\eps(X)+1\leq D\bigl(J_B^\eta(Y)+1\bigr)
	$
	when the distortion is realized.  The general case follows from distortions decreasing to $d_{BM}(X,Y)$ and Theorem~\ref{thm:profile-continuity}.
	
	Subtracting $J_B^\eps(Y)$ and using $J_B^\eps(Y)\leq2$ gives
	\[
	J_B^\eps(X)-J_B^\eps(Y)
	\leq3(D-1)+D\bigl(J_B^\eta(Y)-J_B^\eps(Y)\bigr).
	\]
	Since $0\leq\eta-\eps\leq1-D^{-1}$, Theorem~\ref{thm:profile-continuity} and symmetry yield the uniform estimate.  At $\eps=0$, $\eta=1-D^{-1}<1/2$ for $D<2$, and Lemma~\ref{lem:correction-zero} gives the square-root bound.  For $0<\eps<1/2$, Lemma~\ref{lem:correction-positive} bounds the last difference by $6(\eta-\eps)/\eps$, yielding the fixed-parameter estimate.  When $\eps\geq1/2$, both parameters lie in the constant part of the profile.
\end{proof}

\section{Examples and Exact Computations}\label{sec:examples}

We conclude with examples illustrating the preceding theory. We first
record the exact values in Hilbert and $L_p$ spaces, and then turn to
finite-dimensional examples for which the Birkhoff profile is genuinely
nonconstant.

\begin{example}\label{ex:hilbert}
	Let $H$ be a real Hilbert space with $\dim H\geq2$. Then, for every
	$0\leq\eps<1$,
	\[
	J_I^\eps(H)
	=
	J_B^\eps(H)
	=
	J(H)
	=
	\sqrt2.
	\]
\end{example}

\begin{proof}
	Let $x,y\in S_H$ and write $c=\langle x,y\rangle$.  Then
	$
	\norm{x+y}^2=2+2c,
	\norm{x-y}^2=2-2c.
	$
	Thus
	$
	\min\{\norm{x+y},\norm{x-y}\}^2=2-2|c|\leq2,
	$
	so $J(H)\leq\sqrt2$.  Choosing $x,y\in S_H$ with $\langle x,y\rangle=0$ gives
	$
	\norm{x+y}=\norm{x-y}=\sqrt2,
	$
	so $J(H)=\sqrt2$.  The same pair satisfies exact isosceles orthogonality and exact Birkhoff-James orthogonality; hence it is admissible for $J_I^\eps(H)$ and $J_B^\eps(H)$ for every $\eps$.  The upper bounds $J_I^\eps(H),J_B^\eps(H)\leq J(H)$ complete the proof.
\end{proof}

We next consider $L_p$ and $\ell_p$ spaces, where the relevant extremal pairs are furnished by the classical Clarkson--Hanner geometry.

For $1<p<\infty$ let $q$ be the conjugate exponent, $1/p+1/q=1$.  We use the convention $q=\infty$ for $p=1$ and $q=1$ for $p=\infty$.

\begin{example}\label{ex:approx-lp}
	Let $1\leq p<\infty$, and let $q$ be the conjugate exponent of $p$.
	Suppose that $(\Omega,\Sigma,\mu)$ contains two disjoint measurable
	sets of positive finite measure. Then, for every $0\leq\eps<1$,
	\[
	J_I^\eps(L_p(\mu))
	=
	J_B^\eps(L_p(\mu))
	=
	J(L_p(\mu))
	=
	\max\{2^{1/p},2^{1/q}\}.
	\]
	The same formula holds for $\ell_p$ and for $\ell_p^n$, $n\geq2$.
	For $p=\infty$ and $n\geq2$,
	\[
	J_I^\eps(\ell_\infty^n)
	=
	J_B^\eps(\ell_\infty^n)
	=
	J(\ell_\infty^n)
	=
	2.
	\]
	
	Moreover, if $1<p<\infty$ and $\eps>0$, the value $J(L_p(\mu))$
	can be approached by pairs $u,v\in S_{L_p(\mu)}$ satisfying
	\[
	u\perp_{B,\eps}v
	\qquad\text{but}\qquad
	u\not\perp_B v.
	\]
	Thus the equality above is not merely a consequence of restricting
	the definition to exactly Birkhoff--James orthogonal pairs.
\end{example}

\begin{proof}
	The classical Clarkson--Hanner inequalities give
	$J(L_p(\mu))=\max\{2^{1/p},2^{1/q}\}$.
	Since $J_I^\eps(X)=J(X)$, it remains to consider the
	Birkhoff--James constant.
	
	Let $A,B\in\Sigma$ be disjoint measurable sets of positive finite measure and set
	$
	e_A=\mu(A)^{-1/p}\chi_A,
	e_B=\mu(B)^{-1/p}\chi_B.
	$
	Then $e_A,e_B\in S_{L_p(\mu)}$,
	$\|e_A\pm e_B\|_p=2^{1/p}$, and
	$\|e_A+\lambda e_B\|_p^p=1+|\lambda|^p\geq1$.
	Hence $e_A\perp_{B,\eps}e_B$ and
	$J_B^\eps(L_p(\mu))\geq2^{1/p}$.
	
	Next, put
	$
	x=2^{-1/p}(e_A+e_B),
	y=2^{-1/p}(e_A-e_B).
	$
	Then $x,y\in S_{L_p(\mu)}$, $\|x\pm y\|_p=2^{1/q}$, and
	\[
	\|x+\lambda y\|_p^p
	=
	\frac{|1+\lambda|^p+|1-\lambda|^p}{2}
	\geq1,
	\qquad \lambda\in\mathbb R,
	\]
	since the numerator is minimized at $\lambda=0$.
	Thus $x\perp_{B,\eps}y$, and therefore
	$
	J_B^\eps(L_p(\mu))
	\geq
	\max\{2^{1/p},2^{1/q}\}
	=
	J(L_p(\mu)).
	$
	The reverse inequality is immediate from the definition, so
	$J_B^\eps(L_p(\mu))=J(L_p(\mu))$.
	
	For $1<p<\infty$ and $\eps>0$, one may moreover choose genuinely
	approximate Birkhoff--James pairs. Let $(u,v)$ be one of the preceding
	extremal pairs with $u\perp_Bv$ and $m(u,v)=J(L_p(\mu))$, and define
	$
	v_t=\frac{v+tu}{\|v+tu\|_p}.
	$
	Then $v_t\to v$ and hence $d(u,v_t)\to d(u,v)=1$. Thus, for all
	sufficiently small $t>0$, $u\perp_{B,\eps}v_t$. Since $L_p(\mu)$ is
	smooth, if $f_u$ denotes the norming functional of $u$, then
	$f_u(v)=0$ whereas
	$
	f_u(v_t)=\frac{t}{\|v+tu\|_p}>0.
	$
	Hence $u\not\perp_Bv_t$ for $t>0$. By continuity,
	$m(u,v_t)\to J(L_p(\mu))$. Therefore, for every $\delta>0$, one may
	choose $t>0$ such that
	\[
	u\perp_{B,\eps}v_t,
	\qquad
	u\not\perp_Bv_t,
	\qquad
	m(u,v_t)>J(L_p(\mu))-\delta.
	\]
	
	The arguments for $\ell_p$ and $\ell_p^n$ are identical. For
	$\ell_\infty^n$, take
	$u=(1,1,0,\ldots,0)$ and $v=(1,-1,0,\ldots,0)$. Then
	$\|u\pm v\|_\infty=2$ and
	$
	\|u+\lambda v\|_\infty
	=
	\max\{|1+\lambda|,|1-\lambda|\}
	\geq1,
	$
	so $u\perp_{B,\eps}v$ and consequently
	$J_B^\eps(\ell_\infty^n)=2$.
\end{proof}

\begin{example}\label{ex:mixed-plane}
	Consider the real plane $M=(\mathbb{R}^2,N)$ endowed with the norm
	\[
	N(u,v)
	=
	\max\left\{\sqrt{u^2+v^2},|u-v|\right\}
	=
	\begin{cases}
		\sqrt{u^2+v^2}, & uv\geq0,\\
		|u|+|v|, & uv\leq0.
	\end{cases}
	\]
	Its unit sphere consists of two Euclidean quarter-circles and two line
	segments. This space was considered by Baronti and Papini
	\cite[Example~5.4]{BarontiPapini2022}, where it was shown that
	$J(M)=\sqrt{8/3}$ and the corresponding orthogonal James constant was
	estimated numerically.
	\[
	J_{\perp}(M)<J(M).
	\]
	We use it here to exhibit a genuinely nonconstant approximate
	Birkhoff--James profile.
\end{example}

Set
$
\theta_*=\arctan\frac{1}{\sqrt2},
$
and let $\theta_0\in(\theta_*,\pi/4)$ be the unique solution of
$
\sin\theta_0\bigl(\sin\theta_0+\cos\theta_0\bigr)=\cos\theta_0.
$
For $\theta\in[\theta_*,\theta_0]$, define
$
\sigma(\theta)
=
\frac{\pi}{4}
-
\arccos\left(\frac{\cot\theta}{\sqrt2}\right),
$
$
E(\theta)
=
1-
\frac{\sin(2\theta)}
{\sin(\sigma(\theta)+\theta)},
$
and
$
\eps_*
=
1-\frac{4\sqrt3}{3(1+\sqrt2)}.
$

The preceding example provides a concrete setting in which the
Birkhoff profile is genuinely nonconstant. We now compute this profile
explicitly, determine its exact saturation parameter, and describe its
behavior up to the saturation threshold.
\begin{example}\label{ex:mixed-plane-profile}
	For the space $M$ in Example~\ref{ex:mixed-plane}, the function
	$E$ is a continuous strictly decreasing bijection from
	$[\theta_*,\theta_0]$ onto $[0,\eps_*]$. Moreover,
	\[
	J_I^\eps(M)=J(M)=\sqrt{\frac83},
	\qquad 0\leq\eps<1,
	\]
	and
	\[
	J_B^\eps(M)
	=
	\begin{cases}
		2\cos(E^{-1}(\eps)),
		& 0\leq\eps\leq\eps_*,\\[1mm]
		\sqrt{8/3},
		& \eps_*\leq\eps<1.
	\end{cases}
	\]
	For $0\leq\eps\leq\eps_*$, an extremal pair is
	\[
	x=(\cos\alpha,\sin\alpha),
	\qquad
	y=(\cos\beta,\sin\beta),
	\]
	where
	\[
	\alpha=\sigma(\theta)-\theta,
	\qquad
	\beta=\sigma(\theta)+\theta,
	\qquad
	\theta=E^{-1}(\eps).
	\]
	Hence $J_B^\eps(M)$ is strictly increasing on $[0,\eps_*]$, and
	$\eps_*$ is its smallest saturation parameter. Finally,
	\[
	\lim_{\eps\uparrow\eps_*}
	\frac{\sqrt{8/3}-J_B^\eps(M)}
	{(\eps_*-\eps)^2}
	=
	\frac{\sqrt6(1+\sqrt2)^6}{48}.
	\]
\end{example}

\begin{proof}
	Let
	$
	A=\{(\cos t,\sin t):0\leq t\leq\pi/2\}
	$
	denote the circular part of $S_M$, and let
	$
	L=\{(-t,1-t):0\leq t\leq1\}.
	$
	By symmetry it is enough to consider pairs in $A\cup L$.
	
	Pairs involving $L$ do not produce the James extremum. A direct
	one-variable optimization gives
	$
	\min\{\norm{x+y},\norm{x-y}\}\leq 2\cos\theta_0,
	\text{whenever }x\in L\text{ or }y\in L.
	$
	The value $2\cos\theta_0$ is attained by a coordinate vector and a
	point of the circular arc.
	
	Now take
	$
	x=(\cos\alpha,\sin\alpha),
	y=(\cos\beta,\sin\beta),
	$
	with
	$
	0\leq\alpha\leq\beta\leq\frac{\pi}{2},
	\alpha+\beta\leq\frac{\pi}{2},
	$
	and put
	$
	\theta=\frac{\beta-\alpha}{2},
	\sigma=\frac{\alpha+\beta}{2}.
	$
	Then
	$
	\min\{\norm{x+y},\norm{x-y}\}
	=
	\min\left\{
	2\cos\theta,\,
	2\sin\theta(\sin\sigma+\cos\sigma)
	\right\}.
	$
	Since $\sin\sigma+\cos\sigma\leq\sqrt2$,
	$
	\min\{\norm{x+y},\norm{x-y}\}
	\leq
	\min\{2\cos\theta,2\sqrt2\sin\theta\}
	\leq
	\sqrt{\frac83}.
	$
	Equality is obtained for
	$
	\theta=\theta_*,
	\sigma=\frac{\pi}{4}.
	$
	Hence
	$
	J(M)=\sqrt{\frac83}.
	$
	The same extremal pair is exactly isosceles orthogonal and therefore
	is admissible for every $\eps$, so
	$
	J_I^\eps(M)=J(M)=\sqrt{\frac83}.
	$
	
	For the Birkhoff profile, the directed distance of the above pair is
	$
	d(x,y)
	=
	\frac{\sin(\beta-\alpha)}
	{\max\{\cos\beta,\sin\beta\}}.
	$
	Fix
	$
	r=2\cos t,
	t\in[\theta_*,\theta_0].
	$
	Among all two-arc pairs satisfying $\min\{\norm{x+y},\norm{x-y}\}\geq r$, the largest possible
	distance is attained at
	$
	\alpha=\sigma(t)-t,
	\beta=\sigma(t)+t,
	$
	where
	$
	\sin\sigma(t)+\cos\sigma(t)=\cot t.
	$
	For this pair,
	$
	\min\{\norm{x+y},\norm{x-y}\}=2\cos t
	$
	and
	$
	d(x,y)
	=
	\frac{\sin(2t)}
	{\sin(\sigma(t)+t)}
	=
	1-E(t).
	$
	Consequently, the level $2\cos t$ is admissible for $J_B^\eps(M)$ if
	and only if
	$
	E(t)\leq\eps.
	$
	
	From the definition of $\sigma$,
	$
	\sigma(\theta_*)=\frac{\pi}{4},
	\sigma(\theta_0)=\theta_0,
	$
	and differentiation shows that $E$ is strictly decreasing on
	$[\theta_*,\theta_0]$. Moreover,
	$
	E(\theta_0)=0,
	E(\theta_*)=\eps_*.
	$
	Hence $E$ is a continuous bijection from
	$[\theta_*,\theta_0]$ onto $[0,\eps_*]$, and therefore
	$
	J_B^\eps(M)
	=
	2\cos(E^{-1}(\eps)),
	0\leq\eps\leq\eps_*.
	$
	At $\eps=\eps_*$ the James-extremal pair becomes admissible, so
	$
	J_B^\eps(M)=\sqrt{\frac83},
	\eps_*\leq\eps<1.
	$
	
	Finally, set
	$
	u=\frac{\pi}{4}-\sigma(t).
	$
	Using
	$
	\cot t=\sqrt2\cos u,
	$
	Taylor expansion at $t=\theta_*$ gives
	$
	\sqrt{\frac83}-2\cos t
	=
	\frac{\sqrt6}{9}u^2+o(u^2)
	$
	and
	$
	\eps_*-E(t)
	=
	\frac{4\sqrt3}{3(1+\sqrt2)^3}\,u+o(u).
	$
	Taking the quotient yields
	$
	\lim_{\eps\uparrow\eps_*}
	\frac{\sqrt{8/3}-J_B^\eps(M)}
	{(\eps_*-\eps)^2}
	=
	\frac{\sqrt6(1+\sqrt2)^6}{48}.
	$
\end{proof}

Numerically,
$
J_{\perp}(M)=2\cos\theta_0\approx1.62562,
\eps_*\approx0.04341.
$

\begin{remark}
	For the mixed normed plane $M$, the approximate Birkhoff--James
	endpoint constant need not coincide with the James constant even for
	a positive approximation parameter. Indeed, since
	$\eps_*\approx0.04341475$, taking $\eps=0.02$ gives
	\[
	J_B^{0.02}(M)<J(M)=\sqrt{\frac83}.
	\]
	Thus equality with the James constant may fail before saturation.
\end{remark}

\begin{example}
	\label{ex:same-james-different-profiles}
	Let $M$ be the mixed normed plane in Example~\ref{ex:mixed-plane}, and set
	$
	p_0=\frac{2\log 2}{\log(8/3)},
	Y=\ell_{p_0}^{\,2}.
	$
	Then $1<p_0<2$ and
	$
	2^{1/p_0}=\sqrt{\frac{8}{3}}.
	$
	We show that
	$
	J(M)=J(Y)=\sqrt{\frac{8}{3}},
	$
	whereas
	$
	J_B^\varepsilon(M)<J_B^\varepsilon(Y)
	\text{for every }0\leq\varepsilon<\varepsilon_*,
	$
	with
	$
	\varepsilon_*
	=1-\frac{4\sqrt{3}}{3(1+\sqrt{2})}>0.
	$
	
	Indeed, Example~\ref{ex:mixed-plane-profile} gives
	$J(M)=\sqrt{\frac{8}{3}}$.
	The classical formula for the James constant of $\ell_p^2$ gives
	$
	J(Y)
	=\max\left\{2^{1/p_0},\,2^{1-1/p_0}\right\}
	=2^{1/p_0}
	=\sqrt{\frac{8}{3}},
	$
	where the second equality follows from $p_0<2$.
	
	To compute the Birkhoff profile of $Y$, take
	$u=(1,0)$ and $v=(0,1)$.
	For every $\lambda\in\mathbb R$,
	$
	\|u+\lambda v\|_{p_0}
	=\left(1+|\lambda|^{p_0}\right)^{1/p_0}
	\geq 1=\|u\|_{p_0}.
	$
	Thus $u\perp_B v$, so this pair is admissible for
	$J_B^\varepsilon(Y)$ for every $0\leq\varepsilon<1$.
	Since
	$
	\|u+v\|_{p_0}
	=\|u-v\|_{p_0}
	=2^{1/p_0},
	$
	we obtain
	$
	\sqrt{\frac{8}{3}}
	\leq J_B^\varepsilon(Y)
	\leq J(Y)
	=\sqrt{\frac{8}{3}}.
	$
	Consequently,
	$
	J_B^\varepsilon(Y)=\sqrt{\frac{8}{3}}(0\leq\varepsilon<1).
	$
	
	On the other hand, Example~\ref{ex:mixed-plane-profile} shows that
	$J_B^\varepsilon(M)$ is strictly increasing on
	$[0,\varepsilon_*]$ and first reaches $\sqrt{\frac{8}{3}}$ at
	$\varepsilon=\varepsilon_*$.
	Hence
	$
	J_B^\varepsilon(M)
	<\sqrt{\frac{8}{3}}
	=J_B^\varepsilon(Y)
	(0\leq\varepsilon<\varepsilon_*).
	$
	In particular, for each fixed
	$\varepsilon\in[0,\varepsilon_*)$, the value of
	$J_B^\varepsilon(X)$ cannot be determined by $J(X)$ alone.
	The smallest saturation parameter is $\varepsilon_*$ for $M$
	and zero for $Y$.
\end{example}

The preceding example shows that the James constant alone does not
determine the Birkhoff profile. The next example strengthens this
observation by keeping both endpoint constants and the saturation
parameter fixed.

\begin{example}
	\label{ex:same-two-endpoints}
	Let
	\[
	e=(1,0),\qquad a=\left(\frac{99}{100},\frac1{10}\right),
	\qquad b=\left(\frac45,\frac35\right),
	\qquad r=\left(1,\frac1{20}\right),
	\]
	and write $z^\tau=(z_2,z_1)$ for coordinate interchange. Define two
	centrally symmetric convex bodies by
	\[
	K_X=\operatorname{conv}\{\pm e,\pm a,\pm b,
	\pm b^\tau,\pm a^\tau,\pm e^\tau\},
	\qquad
	K_Y=\operatorname{conv}\bigl(K_X\cup\{\pm r,\pm r^\tau\}\bigr).
	\]
	Let $X$ and $Y$ be the real normed planes whose closed unit balls are
	$K_X$ and $K_Y$, respectively. 
	
	Equivalently, their norms are
	\[
	\|(s,t)\|_X=
	\max\left\{
	\left|s+\frac{t}{10}\right|,
	\frac{|250s+95t|}{257},
	\frac{5|s+t|}{7},
	\frac{|95s+250t|}{257},
	\left|\frac{s}{10}+t\right|,
	|s-t|
	\right\},
	\]
	and
	\[
	\|(s,t)\|_Y=
	\max\left\{
	|s|,
	\frac{|100s+20t|}{101},
	\frac{|250s+95t|}{257},
	\frac{5|s+t|}{7},
	\frac{|95s+250t|}{257},
	\frac{|20s+100t|}{101},
	|t|,
	|s-t|
	\right\}.
	\]
	
	Then
	$
	J(X)=J(Y)=\frac{514}{319},
	J_\perp(X)=J_\perp(Y)=\frac{331}{206}.
	$
	Nevertheless,
	$
	J_B^{1/40}(X)=\frac{331}{206}
	<\frac{373}{232}\leq J_B^{1/40}(Y).
	$
	Thus the two endpoint constants do not determine the profile.
	In fact, these two spaces also have the same saturation parameter,
	\[
	\sigma_B(X)=\sigma_B(Y)=\frac{100999}{3009900}.
	\]
\end{example}

\begin{figure}[htbp]
	\centering
	\includegraphics[width=0.45\textwidth]{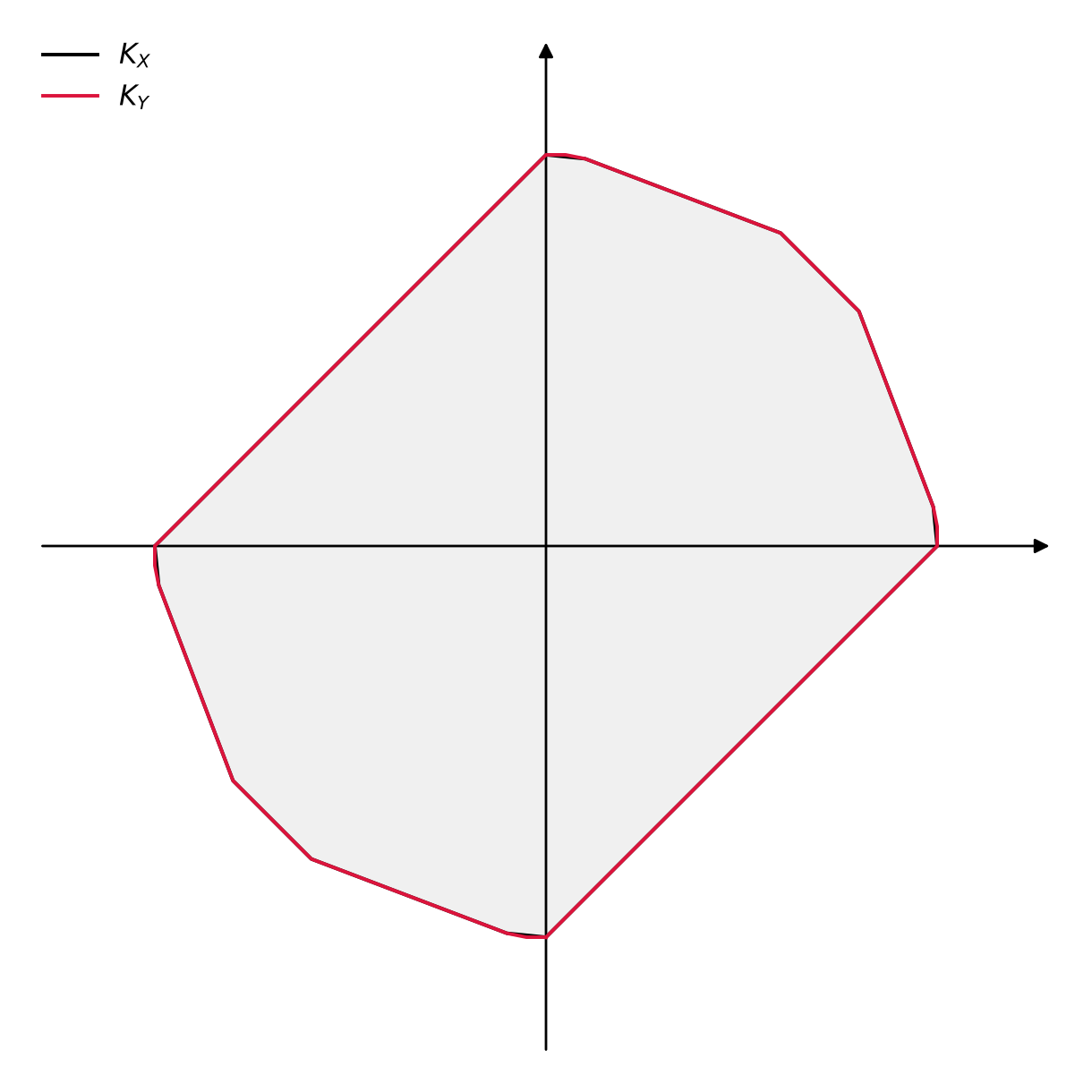}
	\caption{Overlay of the unit balls $K_X$ (black) and $K_Y$ (red, light-gray fill).
		The two bodies coincide everywhere except near the coordinate axes, where
		$K_Y$ bulges outward slightly through the added vertices $\pm r$ and $\pm r^\tau$.
		The inset zooms in on the positive horizontal axis: $K_X$ runs straight from
		$e=(1,0)$ to $a=(99/100,1/10)$, while $K_Y$ passes through $r=(1,1/20)$.}
	\label{fig:unit-balls-overlay}
\end{figure}

\begin{proof}
	For $Z=X,Y$, write
	\[
	m_Z(x,y)=\min\{\|x+y\|_Z,\|x-y\|_Z\},
	\qquad
	d_Z(x,y)=\operatorname{dist}_Z(x,\mathbb Ry).
	\]
	The supporting lines of the two unit balls give
	\begin{align*}
		\|(s,t)\|_X
		&=
		\max\left\{
		\left|s+\frac{t}{10}\right|,
		\frac{|250s+95t|}{257},
		\frac{5|s+t|}{7},
		\frac{|95s+250t|}{257},
		\left|\frac{s}{10}+t\right|,
		|s-t|
		\right\},\\
		\|(s,t)\|_Y
		&=
		\max\left\{
		|s|,
		\frac{|100s+20t|}{101},
		\frac{|250s+95t|}{257},
		\frac{5|s+t|}{7},
		\frac{|95s+250t|}{257},
		\frac{|20s+100t|}{101},
		|t|,
		|s-t|
		\right\}.
	\end{align*}
	
	Both norms are invariant under coordinate interchange and simultaneous
	change of signs. On a product of two polygonal sides, $m_Z$ is
	piecewise affine in the two side parameters. Hence its maximum is
	obtained after reducing to finitely many vertices and intersection
	points of the corresponding affine functions. A finite verification over the relevant vertices and intersection
	points of the affine pieces gives
	$
	J(X)=J(Y)=\frac{514}{319}.
	$
	Moreover, if
	$
	L=\{(-t,1-t):0\leq t\leq1\},
	$
	then the corresponding restricted maximization gives
	$
	\max_{x\in S_Z,\;y\in L} m_Z(x,y)
	=
	\frac{331}{206}, Z=X,Y.
	$
	Hence
	$
	m_Z(x,y)\leq
	c:=\frac{331}{206}
	$
	whenever one of the two vectors belongs to $L\cup(-L)$.
	
	At an interior first-quadrant point, every norming functional has
	nonnegative coordinates, so it cannot annihilate another positive
	vector. Hence, up to symmetry, every exact Birkhoff--James orthogonal
	pair has one vector in $L$. By the preceding estimate,
	$
	J_\perp(X),J_\perp(Y)\leq c.
	$
	For
	$
	x_0=e,
	y_0=
	\left(
	\frac{10846}{35535},
	\frac{64817}{71070}
	\right),
	$
	direct substitution gives, for $Z=X,Y$,
	$
	\|x_0\|_Z=\|y_0\|_Z=1,
	m_Z(x_0,y_0)=c,
	d_Z(x_0,y_0)=1.
	$
	Thus
	$
	J_\perp(X)=J_\perp(Y)=\frac{331}{206}.
	$
	
	It remains to compare the profiles at $\eps=1/40$. For $z$ on the
	first-quadrant part of $S_X$, write $q=z_1-z_2$. Then
	$
	z=
	\left(
	\frac{S(q)+q}{2},
	\frac{S(q)-q}{2}
	\right),
	$
	where
	\[
	S(q)=
	\begin{cases}
		7/5,
		&0\leq |q|\leq1/5,\\
		(514-155|q|)/345,
		&1/5\leq |q|\leq89/100,\\
		(20-9|q|)/11,
		&89/100\leq |q|\leq1.
	\end{cases}
	\]
	The same finite side calculation as above shows that if
	$m_X(x,y)>c$, then, after symmetry, the coordinate differences may be
	written as $u$ and $-v$ with
	\[
	u\geq v,\qquad
	u+v\geq c,\qquad
	S(u)+S(v)\geq\frac75c.
	\]
	For such a pair the larger directed distance is
	$
	D(u,v)
	=
	\frac{S(u)v+uS(v)}{S(v)+v}.
	$
	Checking the two linearity regions of $S$ shows that
	\[
	D(u,v)<\frac{39}{40}
	\]
	throughout the above feasible set. Hence
	\[
	m_X(x,y)>c
	\quad\Longrightarrow\quad
	d_X(x,y)<\frac{39}{40}.
	\]
	Therefore no pair with endpoint minimum greater than $c$ is
	$B$-$1/40$-orthogonal, and the pair $(x_0,y_0)$ yields
	$
	J_B^{1/40}(X)=\frac{331}{206}.
	$
	
	For $Y$, take
	$
	x=
	\left(\frac{199}{200},\frac3{40}\right),
	y=
	\left(\frac{143}{580},\frac{5419}{5800}\right).
	$
	Both vectors belong to $S_Y$, and direct substitution gives
	$
	\|x+y\|_Y=\|x-y\|_Y=\frac{373}{232},
	$
	together with
	$
	d_Y(x,y)
	=
	\frac{1056931}{1083800}
	>
	\frac{39}{40}.
	$
	Thus $x\perp_{B,1/40}y$, and hence
	$
	J_B^{1/40}(Y)
	\geq
	\frac{373}{232}
	>
	\frac{331}{206}
	=
	J_B^{1/40}(X).
	$
	
	Finally, we compute the saturation parameters. The same finite
	maximization shows that the James-extremal pairs in both spaces,
	up to the preceding symmetries, are exactly those satisfying
	$
	u+v=\frac{514}{319},
	\frac15\leq v\leq u\leq\frac{89}{100}.
	$
	On this set the larger directed distance is
	$
	\frac{AJ-2Buv}{A+(1-B)v},
	A=\frac{514}{345},
	B=\frac{31}{69},
	J=\frac{514}{319},
	$
	and is maximal at
	$
	u=\frac{89}{100},
	v=J-\frac{89}{100}.
	$
	For the corresponding extremal pair,
	$
	d_Z(x,y)
	=
	\frac{99}{100}-\frac{709}{30099}, Z=X,Y.
	$
	The finite-dimensional formula for the saturation parameter therefore
	gives
	\[
	\sigma_B(X)=\sigma_B(Y)
	=
	1-
	\left(
	\frac{99}{100}-\frac{709}{30099}
	\right)
	=
	\frac{100999}{3009900}.
	\]
\end{proof}

\section*{Acknowledgements}

Thanks to all the members of the Functional Analysis Research team of the College of Mathematics and Statistics of Anqing Normal University for their discussion and correction of the difficulties and errors encountered in this
paper. The authors cordially thank Professor Hai Zhang for his encouraging advice and the first author would like to acknowledge his strong support for undergraduate	students in carrying out basic mathematics research.

\section*{Declarations}

\subsection*{Competing Interests}
The authors declare that there are no competing interests.

\subsection*{Funding Information}
Not Applicable.

\subsection*{Author contribution}
These authors contributed equally to this work. All authors have read and agreed to the published version of the manuscript.

\subsection*{Data Availability Statement}
Not Applicable.

\end{document}